\documentclass[12pt]{amsart}
\usepackage[latin9]{inputenc}
\usepackage{xcolor}
\usepackage{pdfcolmk}
\usepackage{float}
\usepackage{amsthm}
\usepackage{amssymb}
\usepackage{graphicx}
\usepackage{epstopdf}
\usepackage{esint}
\PassOptionsToPackage{normalem}{ulem}
\usepackage{ulem}

\makeatletter

\providecolor{lyxadded}{rgb}{0,0,1}
\providecolor{lyxdeleted}{rgb}{1,0,0}

\numberwithin{equation}{section}
\numberwithin{figure}{section}

\usepackage{amssymb,latexsym,amsmath,amsthm,graphicx}
\numberwithin{equation}{section}
\newtheorem{thm}{Theorem}[section]
\newtheorem{prop}[thm]{Proposition}

\newtheorem*{remark}{Remark}

\newcommand{\pd}{\partial}

\newcommand{\DD}{\mathbb{D}}
\newcommand{\CC}{\mathbb{C}}
\newcommand{\BAC}{\lambda_{A^{2}}}

\newcommand{\dz}{\partial_{z}}

\makeatother

\begin{document}

\title[Bergman analytic content]{The Bergman analytic content of planar domains}

\author{Matthew Fleeman and Erik Lundberg}
\begin{abstract}
Given a planar domain $\Omega$, the \emph{Bergman analytic content}
measures the $L^{2}(\Omega)$-distance between $\bar{z}$ and the
Bergman space $A^{2}(\Omega)$. 
We compute the Bergman analytic content
of simply-connected quadrature domains with quadrature formula supported at one point,
and we also determine the function $f \in A^2(\Omega)$ that best approximates $\bar{z}$.
We show that,
for simply-connected domains, the square of Bergman analytic content
is equivalent to \emph{torsional rigidity} from classical elasticity
theory, while for multiply-connected domains these two domain constants
are not equivalent in general. 
\end{abstract}

\maketitle

\section{Introduction}

Recall that, for a bounded domain $\Omega$ in $\mathbb{C}$, the
Bergman space $A^{2}(\Omega)$ is the Hilbert space of functions holomorphic
in $\Omega$ that satisfy $\int_{\Omega}\vert f(z)\vert^{2}dA(z)<\infty.$
Extending previous studies on the approximation of $\bar{z}$ in analytic
function spaces, in \cite{GuadKhav} the authors introduced the notion
of \emph{Bergman analytic content}, defined as the $L^{2}(\Omega)$-distance
between $\bar{z}$ and the space $A^{2}(\Omega)$. They showed that
the best approximation to $\bar{z}$ is $0$ if and only if $\Omega$
is a disk, and that the best approximation is $\frac{c}{z}$ if and
only if $\Omega$ is an annulus centered at the origin.

D. Khavinson and the first author in \cite{FleemanKhavinson} reduced
the problem of finding the best approximation to that of solving a
Dirichlet problem in $\Omega$ with boundary data $\left|z\right|^{2}.$
Using this fact, we will determine the best approximation and Bergman
analytic content when $\Omega$ is a simply-connected, one-point quadrature
domain.  

Recall that a bounded domain $\Omega \subset \CC$ is called a \emph{quadrature domain} if it admits a formula expressing the area integral of any test function $g \in A^2(\Omega)$
as a finite sum of weighted point evaluations of $g$ and its derivatives:
\begin{equation}\label{eq:QF1}
 \int_{D} {g(z) dA(z)} = \sum_{m=1}^{N} \sum_{k=0}^{n_m}{a_{m,k}g^{(k)}(z_m)},
\end{equation}
where the points  $z_m \in \Omega$ and constants $a_{m,k}$ are each independent of $g$.
This class of domains is $C^{\infty}$-dense
in the space of domains having a $C^{\infty}$-smooth boundary \cite[Thm. 1.7]{Bell},
and the restricted class of quadrature domains for which $N=1$ in (\ref{eq:QF1}) has the same density property.
When $\Omega$ is a simply-connected quadrature domain with $N=1$,
the conformal mapping $\phi: \DD \rightarrow \Omega$ is a polynomial,
and by making a translation we may assume that the quadrature formula is supported at $\phi(0) = 0$.
\begin{thm}
\label{thm:compBAC} 
Let $\Omega\subset\CC$ be a simply-connected quadrature domain with quadrature formula supported at a single point (say the origin), 
and let $\phi:\DD\rightarrow\Omega$ be the (polynomial)
conformal map from the unit disk 
\[
\phi(z)=\sum_{k=1}^{n}a_{k}z^{k}.
\]
Then the Bergman analytic content $\lambda_{A^{2}}(\Omega)$ of $\Omega$
is given by: 
\begin{equation}
\lambda_{A^{2}}(\Omega)=\pi^{1/2}\left[\sum_{m=1}^{2n-1}\frac{\left|c_{m}\right|^{2}}{m+1}-\sum_{k=1}^{n-1}k\left|\sum_{j=1}^{n-k}a_{k+j}\overline{a_{j}}\right|^{2}\right]^{1/2},\label{eq:QDexplicit}
\end{equation}
where 
\[
c_{m}:=\sum_{k+j=m+1}ka_{k}a_{j}\qquad1\leq k,j\leq n.
\]
Moreover, the best approximation $f$ to $\bar{z}$ is the derivative
$f=F'$ of $F=P\circ\phi^{-1},$ where 
\[
P(\zeta)=\frac{1}{2}\sum_{k=1}^{n}\left|a_{k}\right|^{2}+\sum_{k=1}^{n-1}\sum_{j=1}^{n-k}a_{k+j}\overline{a_{j}}\zeta^{k}.
\]
\end{thm}
In \cite{FleemanKhavinson}, the authors characterized domains for
which the best approximation to $\bar{z}$ is a polynomial, showing
in particular that the only quadrature domain with this property is
a disk. On the other hand, Theorem \ref{thm:compBAC} reveals that
for a large class of quadrature domains the best approximation $f$
has a primitive $F$ that becomes a polynomial in the right coordinate
system.

Let $\Omega$ be a domain in the plain, bounded by finitely many Jordan
curves $\Gamma_{0},\ldots,\Gamma_{n}$, with $\Gamma_{0}$ the outer
boundary curve. Then the \emph{torsional rigidity} of $\Omega$ equals
\begin{equation}
\rho(\Omega):=\int_{\Omega}\left|\nabla\nu\right|^{2}dA,\label{eq:TorsionalRigidity_Bandle}
\end{equation}
where $\nu$ solves the Dirichlet problem 
\[
\begin{cases}
\Delta\nu=-2 & \mathrm{in}\;\Omega\\
\nu|_{\Gamma_{0}}=0\\
\nu|_{\Gamma_{i}}=c_{i} & i=1,\ldots,n
\end{cases},
\]
where the constants $c_{i}$ are not known \emph{a priori} but are
determined from the conditions 
\[
\int_{\Gamma_{i}} \partial_n \nu ds =2a_{i},\qquad i=1,\ldots,n,
\]
where $\partial_n$ denotes differentiation in the direction of the outward normal,
$ds$ is the arclength element,
and $a_{i}$ is the area enclosed by $\Gamma_{i}$ 
(cf. \cite[pp. 63-66]{Bandle}).
The function $\nu$ is referred to as the ``stress function'' in elasticity theory,
and the torsional rigidity measures the resistance to twisting
of a cylindrical beam with cross section $\Omega$.

In \cite{FleemanKhavinson} the inequality 
\begin{equation}
\sqrt{\rho(\Omega)}\leq\lambda_{A^{2}}(\Omega).\label{eq:BACtorsion}
\end{equation}
was shown to hold for simply-connected domains. We show that this
is an equality for simply-connected domains. 
\begin{thm}
\label{thm:MainEquality} Suppose $\Omega$ is a bounded, simply-connected
domain. Then $\lambda_{A^{2}}(\Omega)^{2}=\rho(\Omega)$. 
\end{thm}
Theorem \ref{thm:MainEquality} led us to notice that
some of the methods and examples in the current paper overlap with
classical studies of torsional rigidity \cite[Ch. 22]{Musk}, and
while the explicit formulas in Theorem \ref{thm:compBAC} appear to
be new, our proof based on conformal mapping is very similar to the
procedure described in \cite[Sec. 134]{Musk}.
The square of Bergman analytic content
is not in general equivalent to torsional rigidity. This follows from
explicit computations for doubly-connected domains such as the annulus
(see Section \ref{sec:examples}). 

As an extension of Theorem \ref{thm:compBAC},
it would be interesting to determine the best approximation to $\bar{z}$
and the Bergman analytic content for general quadrature domains.   
Since the problem again reduces to solving a Dirichlet problem with data $z\bar{z}$,
the procedure developed in \cite{Bell2}
for solving the Dirichlet problem with rational (in $z$ and $\bar{z}$) data
seems promising.

\smallskip

\noindent \textbf{Outline.} We prove Theorem \ref{thm:compBAC} in
Section \ref{sec:proof1} and Theorem \ref{thm:MainEquality} in Section
\ref{sec:proof2}. We discuss examples in Section \ref{sec:examples}.
In Section \ref{sec:poly}, we revisit the class of domains considered
in \cite{FleemanKhavinson} for which the best approximation to $\bar{z}$
is a monomial.

\smallskip

\noindent {\bf Acknowledgement.}
We wish to thank Dmitry Khavinson for helpful discussions and valuable feedback.
We would also like to acknowledge Jan-Fredrik Olsen who expressed Theorem \ref{thm:MainEquality} as a conjecture
during conversations at the conference ``Completeness problems, Carleson measures, and spaces of analytic functions'' at Mittag-Leffler.

\section{Proof of Theorem \ref{thm:compBAC}}

\label{sec:proof1}

By the definition of Bergman analytic content, we have $\lambda_{A^{2}}(\Omega)=\left\Vert \overline{z}-f\right\Vert _{2}$,
where $f$ is the projection of $\overline{z}$ onto $A^{2}(\Omega)$.
By the Pythagorean theorem we then have that 
\begin{align*}
\BAC(\Omega) & =\left(\int_{\Omega}\left|\overline{z}\right|^{2}dA(z)-\int_{\Omega}\left|f(z)\right|^{2}dA(z)\right)^{1/2}\\
 & =\left(\int_{\DD}\left|\overline{\phi}\phi'\right|^{2}dA-\int_{\DD}\left|f\circ\phi\right|^{2}\left|\phi'\right|^{2}dA\right)^{1/2},
\end{align*}
where we have changed variables $z=\phi(\zeta)$, $dA(z)=|\phi'(\zeta)|^{2}dA(\zeta)$.
The first term $\int_{\DD}\left|\overline{\phi}\phi'\right|^{2}dA=\int_{\DD}\left|\phi\phi'\right|^{2}dA$
is simply the square of the Bergman norm of a polynomial $\phi\phi'$:
\begin{equation}
\int_{\DD}\left|\overline{\phi}\phi'\right|^{2}dA=\pi\sum_{m=1}^{2n-1}\frac{\left|c_{m}\right|^{2}}{m+1},\label{eq:z-bar norm}
\end{equation}
where 
\[
c_{m}:=\sum_{k+j=m+1}ka_{k}a_{j}\qquad1\leq k,j\leq n,
\]
are the coefficients in the expansion of the product $\phi\cdot\phi'$.

In order to compute $\int_{\Omega}\left|f(z)\right|^{2}dA(z)$, we
first find $f$ explicitly. By \cite[Thm. 1]{FleemanKhavinson}, $f=F'$,
where $u=\mathrm{Re}(F)$ solves the Dirichlet problem 
\[
\begin{cases}
\Delta u & =0\\
u|_{\pd\Omega} & =\frac{\left|z\right|^{2}}{2}
\end{cases}.
\]
Changing coordinates using the conformal map $\phi$, we obtain a
harmonic function $\tilde{u}=u\circ\phi$ that solves the following
Dirichlet problem in the unit disk: 
\[
\begin{cases}
\Delta\tilde{u} & =0\\
\tilde{u}|_{\mathbb{T}} & =\frac{\phi\overline{\phi}}{2}
\end{cases}.
\]
Now, on $\mathbb{T}$ we have that $\phi\overline{\phi}=P(\zeta)+\overline{P(\zeta)}$,
where 
\[
P(\zeta)=\frac{1}{2}\sum_{k=1}^{n}\left|a_{k}\right|^{2}+\sum_{k=1}^{n-1}\sum_{j=1}^{n-k}a_{k+j}\overline{a_{j}}\zeta^{k}.
\]
Since this is a harmonic polynomial, we have that $\tilde{u}(\zeta)=\mathrm{Re}(P(\zeta))$.
Thus, $F\circ\phi=P$, and so by the chain rule $(f\circ\phi)(\phi')=p$,
where 
\[
p(\zeta)=P'(\zeta)=\sum_{k=1}^{n-1}k\sum_{j=1}^{n-k}a_{k+j}\overline{a_{j}}\zeta^{k-1}.
\]
Calculating the Bergman norm of this polynomial, we find that 
\begin{equation}
\int_{\Omega}\left|f(z)\right|^{2}dA(z)=\int_{\DD}\left|f\circ\phi\right|^{2}\left|\phi'\right|^{2}dA=\sum_{k=1}^{n-1}k\left|\sum_{j=1}^{n-k}a_{k+j}\overline{a_{j}}\right|^{2}.\label{eq:Projection norm}
\end{equation}
Combining \eqref{eq:z-bar norm} and \eqref{eq:Projection norm},
the result follows.

\section{Proof of Theorem \ref{thm:MainEquality}}

\label{sec:proof2}

Recall that if $\Omega$ is a simply connected domain, the torsional
rigidity $\rho=\rho(\Omega)$ is given by equation (\ref{eq:TorsionalRigidity_Bandle})
\begin{equation}
\rho=\int_{\Omega}|\nabla\nu|^{2}dA,\label{eq:torsion}
\end{equation}
where $\nu$ is the unique solution to the Dirichlet problem 
\[
\begin{cases}
\Delta\nu & =-2\\
\nu\vert_{\partial\Omega} & =\;0
\end{cases}.
\]

Consider the function $u(z):=\nu(z)+\frac{\left|z\right|^{2}}{2}$.
Then $u$ solves the Dirichlet problem stated in the proof of Theorem
\ref{thm:compBAC}: 
\[
\begin{cases}
\Delta u & =0\\
u|_{\pd\Omega} & =\frac{\left|z\right|^{2}}{2}
\end{cases}.
\]
Thus, $u=\mathrm{Re}(F)$, where $f=F'$ is the best approximation
to $\bar{z}$.

Letting $\nu$ denote the torsion function, we have: 
\begin{align*}
\rho(\Omega) & =\int_{\Omega}|\nabla\nu|^{2}dA\\
 & =\int_{\Omega}|2\dz\nu|^{2}dA\\
 & =\int_{\Omega}\left|2\dz u-2\dz\frac{|z|^{2}}{2}\right|^{2}dA\\
 & =\int_{\Omega}\left|F'-\bar{z}\right|^{2}dA\\
 & =\int_{\Omega}\left|\bar{z}-f\right|^{2}dA\\
 & =\int_{\Omega}|z|^{2}-|f|^{2}dA\\
 & =\lambda_{A^{2}}(\Omega)^{2},
\end{align*}
and this completes the proof.

\section{Examples}

\label{sec:examples}

\subsection{Epicycloids}

Let us consider the one-parameter family of domains $\Omega$ with
conformal map $\phi:\DD\rightarrow\Omega$, given by $\phi(z)=z+az^{n}$,
with $0\leq a\leq1/n$.

Applying Theorem \ref{thm:compBAC}, we immediately obtain: 
\[
\lambda_{A^{2}}(\Omega)=\sqrt{\frac{\pi\left(1+4a^{2}+na^{4}\right)}{2}}.
\]

When $a=1/n$ the domain develops cusps (the case $n=4$ is plotted
in Figure \ref{fig:epi}). The case $n=2$ and $a=1/2$ is a cardioid
(cf. \cite[Sec. 58]{Sokolnikoff}).

\begin{figure}[H]
\centering \includegraphics[scale=0.3]{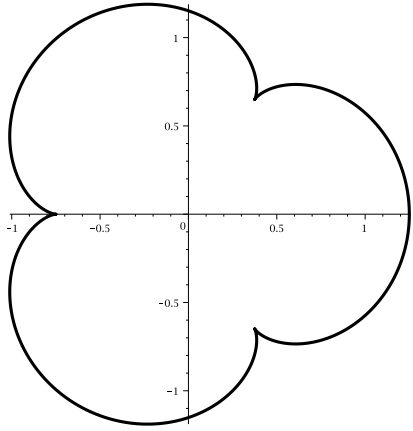}
\caption{The epicycloid domain when $n=4$, $a=1/4$.}
\label{fig:epi} 
\end{figure}

\subsection{The annulus}

The following example shows that Theorem \ref{thm:MainEquality} does
not hold in general for multiply-connected domains. Let $\Omega=\{z:\:r<\left|z\right|<R\}$
be the annulus. The best approximation to $\bar{z}$ in $A^{2}(\Omega)$
is $f(z)=\frac{C}{z}$, where 
\[
C=\frac{R^{2}-r^{2}}{2(\log R-\log r)}
\]
(cf. \cite{FleemanKhavinson} and \cite{GuadKhav}). Following the
proof in Section \ref{sec:proof1}, we have that 
\begin{equation}
\BAC(\Omega)^{2}=\int_{\Omega}|z|^{2}-|f|^{2}dA.\label{eq:BAC}
\end{equation}
Integrating in polar coordinates we get that 
\[
\int_{\Omega}|z|^{2}dA=\frac{\pi}{2}(R^{4}-R^{2}),
\]
and 
\begin{align*}
\int_{\Omega}\left|\frac{C}{z}\right|^{2}dA & =2\pi C^{2}\int_{r}^{R}\frac{1}{\rho}d\rho\\
 & =\frac{\pi}{2}\frac{(R^{2}-r^{2})^{2}}{\log R-\log r}.
\end{align*}
Thus, we have that 
\[
\BAC(\Omega)^{2}=\frac{\pi}{2}\left(\left(R^{4}-r^{4}\right)-\frac{(R^{2}-r^{2})^{2}}{\log R-\log r}\right),
\]
which is smaller than the torsional rigidity \cite[p. 64]{Bandle}
of $\Omega$: 
\[
\rho(\Omega)=\frac{\pi}{2}\left(R^{4}-r^{4}\right).
\]
So we find that neither Theorem \ref{thm:MainEquality} nor the inequality
(\ref{eq:BACtorsion}) hold for multiply-connected domains.

\subsection{The annular region bounded by a pair of confocal ellipses}

We consider the region $G$ between two confocal ellipses that is
the image of an annulus $\Omega:=\{z\in\CC:r<|z|<R\}$ under a Joukowski
map $\phi(z)=z+\frac{1}{z}$.

\begin{figure}[H]
\centering \includegraphics[scale=0.4]{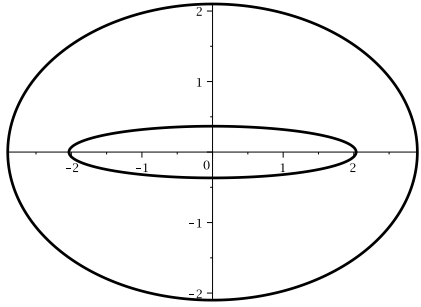}
\caption{The annular region $G$ when $r=1.2$, $R=2.5$.}
\label{fig:confocal} 
\end{figure}

Following the proof in Section \ref{sec:proof1}, the projection of
$\bar{z}$ to the Bergman space is given by $f=F'$, where $u=\mathrm{Re}(F)$
solves the Dirichlet problem 
\[
\begin{cases}
\Delta u & =0\\
u|_{\pd G} & =\frac{\left|z\right|^{2}}{2}
\end{cases}.
\]
The function $\tilde{u}=u\circ\phi$ is harmonic and solves the following
Dirichlet problem in the annulus $\Omega:=\{\zeta\in\CC:r<|\zeta|<R\}$:
\[
\begin{cases}
\Delta\tilde{u} & =0\\
\tilde{u}|_{\partial\Omega} & =\frac{\phi\overline{\phi}}{2}
\end{cases}.
\]
We make the ansatz 
\[
2\tilde{u}(\zeta)=A+B\log|\zeta|+C(\zeta^{2}+\bar{\zeta}^{2})+D\left(\frac{1}{\zeta^{2}}+\frac{1}{\bar{\zeta}^{2}}\right).
\]
The boundary condition gives: 
\[
2\tilde{u}(\zeta)=|\zeta|^{2}+\frac{1}{|\zeta|^{2}}+\frac{\zeta}{\bar{\zeta}}+\frac{\bar{\zeta}}{\zeta},\quad\text{ on \ensuremath{\partial\Omega}}.
\]
Using polar coordinates to parameterize the two circular boundary
components $z=re^{i\theta}$ and $z=Re^{i\theta}$, we obtain two equations: 
\begin{align*}
A+B\log r+2\left(Cr^{2}+\frac{D}{r^{2}}\right)\cos(2\theta) & =r^{2}+\frac{1}{r^{2}}+2\cos(2\theta),\\
A+B\log R+2\left(CR^{2}+\frac{D}{R^{2}}\right)\cos(2\theta) & =R^{2}+\frac{1}{R^{2}}+2\cos(2\theta),
\end{align*}
which implies the system of equations for $A,B,C,D$ 
\begin{align*}
Cr^{2}+\frac{D}{r^{2}} & =1,\\
CR^{2}+\frac{D}{R^{2}} & =1,\\
A+B\log R & =R^{2}+\frac{1}{R^{2}},\\
A+B\log r & =r^{2}+\frac{1}{r^{2}}.
\end{align*}
Solving this (linear in $A,B,C,D$) system, we obtain: 
\begin{align*}
A & =\frac{-\log r}{\log R-\log r}\left(R^{2}+\frac{1}{R^{2}}\right)+\frac{\log R}{\log R-\log r}\left(r^{2}+\frac{1}{r^{2}}\right),\\
B & =\frac{1}{\log R-\log r}\left(R^{2}+\frac{1}{R^{2}}-r^{2}-\frac{1}{r^{2}}\right),\\
C & =\frac{1}{R^{2}+r^{2}},\\
D & =\frac{r^{2}R^{2}}{R^{2}+r^{2}}.
\end{align*}
We have 
\[
(f\circ\phi)\phi'=\frac{B}{2\zeta}+C\zeta-\frac{D}{\zeta^{3}},
\]
and thus the square of the Bergman norm of $f$ is 
\begin{align*}
\int_{G}\left|f(z)\right|^{2}dA(z) & =\int_{\Omega}\left|f\circ\phi\right|^{2}\left|\phi'\right|^{2}dA\\
 & =\int_{\Omega}\left|\frac{B}{2\zeta}+C\zeta-\frac{D}{\zeta^{3}}\right|^{2}dA\\
 & =\frac{\pi}{2}\left(B^{2}(\log R-\log r)+C^{2}(R^{4}-r^{4})+D^{2}\left(\frac{1}{r^{4}}-\frac{1}{R^{4}}\right)\right).
\end{align*}
The square of the Bergman norm of $\bar{z}$ is 
\begin{align*}
\int_{G}\left|z\right|^{2}dA(z) & =\int_{\Omega}\left|\phi\phi'\right|^{2}dA\\
 & =\int_{\Omega}\left|\left(\zeta+\frac{1}{\zeta}\right)\left(1-\frac{1}{\zeta^{2}}\right)\right|^{2}dA(\zeta)\\
 & =\int_{\Omega}\left|\zeta-\frac{1}{\zeta^{3}}\right|^{2}dA(\zeta)\\
 & =\frac{\pi}{2}\left(R^{4}-r^{4}+\frac{1}{r^{4}}-\frac{1}{R^{4}}\right).
\end{align*}
Thus, $\lambda_{A^{2}}(G)^{2} = \int_{G}\left|z\right|^{2}dA(z) - \int_{G}\left|f(z)\right|^{2}dA(z)$ is given by: 
\[
 \frac{\pi}{2}\left(R^{4}-r^{4}+\frac{1}{r^{4}}-\frac{1}{R^{4}}-\frac{1}{\log R-\log r}\left(R^{2}+\frac{1}{R^{2}}-r^{2}-\frac{1}{r^{2}}\right)^{2} - 2\frac{R^{2}-r^{2}}{R^{2}+r^{2}}\right).
\]

\section{Domains for which the best approximation is a monomial}

\label{sec:poly}

The domains defined by $C\mathrm{Re}(z^{n})-\left|z\right|^{2}+1>0$
represent an interesting class of examples (cf. \cite{FleemanKhavinson}).  
These are the domains for which the best approximation to $\bar{z}$ is a monomial,
namely, $\frac{C n}{2} z^{n-1}$.
However, as indicated in Figure \ref{fig:broken},
there are values of $C$ for which the set $\{z:C\mathrm{Re}(z^{n})-\left|z\right|^{2}+1>0\}$
does not have a bounded component, and $\bar{z}$ is no longer in $L^2(\Omega)$. 
Here we address the question of what range of $C$
leads to a bounded component.

\begin{figure}[H]
\includegraphics[scale=0.3]{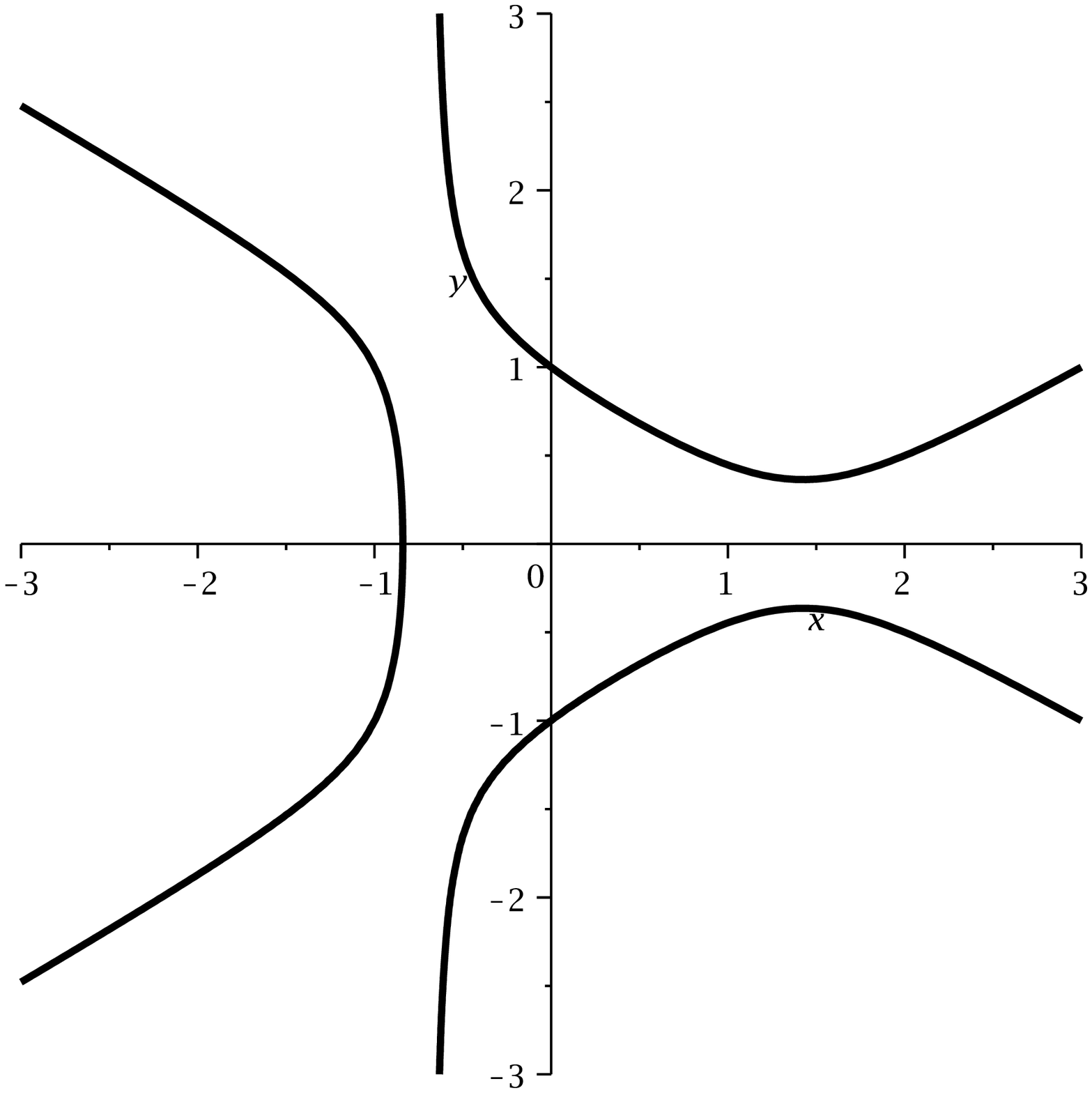}
\caption{The region $\{z:\frac{1}{2}\mathrm{Re}(z^{3})-\left|z\right|^{2}+1>0\}$
does not have a bounded component.}
\label{fig:broken} 
\end{figure}

\begin{prop}
The set $\{z:C\mathrm{Re}(z^{n})-\left|z\right|^{2}+1>0\}$ has
a bounded component whenever 
\[
C\leq\frac{2(n-2)^{\frac{n-2}{2}}}{n^{\frac{n}{2}}}.
\]
\end{prop}
\begin{proof}
Take $z=re^{i\theta}$ and let $f(r,\theta):=C\cos(n\theta)r^{n}-r^{2}+1$
be the defining function of the domain in polar coordinates. We will
show that when 
\[
C\leq\frac{2(n-2)^{\frac{n-2}{2}}}{n^{\frac{n}{2}}}
\]
we have $f(R,\theta)\leq0$ for all $\theta$, where $R=(\frac{2}{nC})^{1/(n-2)}$.
Since the region $\{z:C\mathrm{Re}(z^{n})-\left|z\right|^{2}+1>0\}$
clearly contains the origin, this ensures that it has a component
entirely contained in the disk $|z|<R$.

It is enough to show that $f(R,0)\leq0$ since we have $f(R,\theta)\leq f(R,0)$.

The function $F(r):=f(r,0)=Cr^{n}-r^{2}+1$, has derivative $F'(r)=Cnr^{n-1}-2r$,
with a critical point at $R=(\frac{2}{nC})^{1/(n-2)}$, which by the
first derivative test is a local minimum. Plugging this critical point
into $F$, we find that 
\[
C\left(\frac{2}{nC}\right)^{n/(n-2)}-\left(\frac{2}{nC}\right)^{2/(n-2)}+1\leq0
\]
precisely when 
\[
C\leq\frac{2(n-2)^{\frac{n-2}{2}}}{n^{\frac{n}{2}}}.
\]
\end{proof}

\end{document}